\documentclass[a4paper,11pt]{article}
\pdfoutput=1
\usepackage{graphicx}
\usepackage{cite}
\usepackage{picinpar}
\usepackage{amsmath}
\usepackage{url}
\usepackage{flushend}
\usepackage[latin1]{inputenc}
\usepackage{colortbl}
\usepackage{soul}
\usepackage{multirow}
\usepackage{pifont}
\usepackage{color}
\usepackage{alltt}
\usepackage[hidelinks]{hyperref}
\usepackage{enumerate}
\usepackage{siunitx}
\usepackage{breakurl}
\usepackage{epstopdf}
\usepackage{pbox}
\usepackage{subfigure}
\usepackage{amssymb}\usepackage{indentfirst}
\usepackage[thmmarks,amsmath]{ntheorem}
\begin{document}
\theorembodyfont{\normalfont}
\theoremheaderfont{\itshape}
\theoremseparator{:}
\newtheorem{theorem}{\indent Theorem}
\newtheorem{lemma}{\indent Lemma}
\newtheorem{proposition}{\indent Proposition}
\newtheorem{corollary}{\indent Corollary}
\newtheorem{definition}{\indent Definition}
\newtheorem{remark}{\indent Remark}
\newtheorem{example}{\indent Example}
\newtheorem*{proof}{Proof}
\def\QEDclosed{\mbox{\rule[0pt]{1.5ex}{1.5ex}}} 
\def\QED{\QEDclosed} 
\def\endproof{\hspace*{\fill}~\QED\par\endtrivlist\unskip}
\newcommand\blfootnote[1]{%
\begingroup
\renewcommand\thefootnote{}\footnote{#1}%
\addtocounter{footnote}{-1}%
\endgroup
}
\title{The Controllability and Structural Controllability of Laplacian Dynamics}
\author{Jijun Qu, Zhijian Ji, Yungang Liu, and Chong Lin
\thanks{E-mail address: jizhijian@pku.org.cn. (Zhjian Ji)
This work was supported by the National Natural Science Foundation of China (Grant Nos. 61873136 and 62033007), Taishan Scholars Climbing Program of Shandong Province of China and Taishan Scholars Project of Shandong Province of China (No. ts20190930).
Jijun Qu, Zhijian Ji (Corresponding author) and Chong Lin are with Institute of Complexity Science, College of Automation, Qingdao University, and Shandong Key Laboratory of Industrial Control Technology. Yungang Liu is with School of Control Science and Engineering, Shandong University.}}\maketitle
\begin{abstract}
In this paper, classic controllability and structural controllability under two protocols are investigated. For classic controllability, the multiplicity of eigenvalue zero of general Laplacian matrix $L^*$  is shown to be determined by the sum of the numbers of zero circles, identical nodes and opposite pairs, while it is always simple for the Laplacian $L$ with diagonal entries in absolute form. For a fixed structurally balanced topology,  the controllable subspace is proved to be invariant even if the antagonistic weights are selected differently under the corresponding protocol with $L$.   For a graph expanded from a star graph rooted from a single leader, the dimension of controllable subspace is  two under the protocol associated with $L^*$. In addition, the system is structurally controllable under both protocols if and only if the topology without unaccessible nodes is connected. As a reinforcing case of structural controllability, strong structural controllability requires the system to be controllable for any choice of weights. The connection between father nodes and child nodes affects  strong structural controllability because it determines the linear relationship of the control information from father nodes. This discovery is a major factor in establishing the sufficient conditions on strong structural controllability for multi-agent systems under both protocols, rather than for complex networks, about latter results are already abundant.
\end{abstract}
\section{Introduction}
Controllability is a research topic that attracts more and more attention in multi-agent system, which directly promotes the control and design of the system. An important origin of multi-agent systems is biological groups, such as birds, fish, bees and ant colonies, etc. The dynamic evolution of such systems is closely related to the information transmission relationship among agents. In particular, the distributed controllability based on neighbors ensures that the network system has the ability to obtain behavior goals, such as consensus, formation control and stabilization \cite{newcyb,conliu1,consensusliu,fc,stab2}, etc.

The controllability of multi-agent systems was first considered to general networks \cite{tanner}. It is known that the upper and lower bounds of controllable subspace can be determined by almost equitable partition and distance partition, respectively \cite{cesar,distance,zhangshuo}. Special topologies-path, circle and star graphs were concerned in \cite{pc,cao} which provided sufficient and necessary conditions for controllability. The antagonistic network was first considered by Altafini \cite{biconsensus}, in which the concept of structural balance was proposed for consensus problems. The controllability of antagonistic networks was also considered in \cite{sun}, wherein an equivalence was developed for the controllability of antagonistic networks and general networks. The controllability of signed networks was developed in \cite{signed1,signed2,signed3,signed4}. Equitable partitions and symmetric topologies were used to analyze the controllability in \cite{jimeng}.

The concept of structural controllability was proposed by Lin, who gave the necessary and sufficient conditions for complex networks to have this basic property \cite{lin}. A system is structurally controllable if and only if there is at least one choice of weights such that it is controllable. The structural controllability of multi-agent systems was also considered in \cite{sc}. Inspired by the above work, the developed work of structural controllability for multi-agent systems was  exhibited in \cite{sc1}, which, however, dose not pay enough attention to the relationship between complex networks and multi-agent systems. It is known from \cite{lin,sc,sc1,scguan} that the existence of self-loops makes the topology of multi-agent system free of dilation, which makes the controllability necessary condition of $rank(L;B)=n$ always true, while complex networks do not have this feature. Therefore, the controllability and strong structural controllability of multi-agent systems are easier to realize than complex networks. Based on the work of structural controllability, strong structural controllability was concerned in \cite{ssc,zeroforce1,zeroforce2}.  As a reinforcing case of structural controllability, strong structural controllability requires the system to be controllable for any choice of weights. By color change rule, if all nodes are black, the system is strongly structurally controllable, and the leader set is called as a zero forcing set \cite{zeroforce1}. It was pointed that zero forcing set is a special case of balancing set which is in terms of the connection between father nodes and child nodes to construct a strongly structurally controllable system \cite{zeroforce2}. Although zero forcing set and balancing set can provide sufficient and necessary conditions for strong structural controllability, unfortunately, under the two protocols discussing in this paper, these sets only arise sufficient conditions. This means that protocols bring a non-negligible impact on controllability. The strong structural controllability of multi-agent systems, signed networks and linear systems was also considered in \cite{com1,com2,com3,com4,com5,com6}.

Under the two protocols, classic controllability, structural controllability and strong structural controllability may be inconsistent or unified, which is the main issue discussed in this paper. Below are the contributions.
 \begin{itemize}
   \item The classical controllability under different protocols is proved to be different, which is due to the different multiplicity of the zero eigenvalue of Laplacian matrix under different protocols. We establish a necessary and sufficient condition for the multiplicity of zero eigenvalue, which leads to a sufficient condition for the uncontrollability of the system. 
   \item Necessary and sufficient conditions are derived for the dimension of controllable subspace under two different protocols. In addition, it is proved that the controllable subspace is invariant even for different choices of antagonistic weights.
   \item A unified result of structural controllability under the two protocols is given in the form of necessary and sufficient conditions, which shows that the system is structurally controllable if and only if the weighted graph with no unaccessible node is connected.
   \item We give a lower bound on the dimension of the controllable subspace of strong structural controllability under the two protocols and two sufficient conditions for the strong structural controllability.
 \end{itemize}

The rest of this paper is documented as follows.  In Section 2, the graph theory and two different models are presented. Section 3 shows the affection of two protocols on controllability.  The structural controllability of weighted graphs is considered in Section 4.  Finally, the conclusions are arranged in Section 5.
\section{The Model of Laplacian Dynamics}
By a graph we mean a pair $\mathcal{G}=\{\mathcal{V},\mathcal{E}\}$ consisting of node set $\mathcal{V}$ and edge set $\mathcal{E}$, where $\mathcal{V}=\{v_1,\ldots,v_n\}$, and $\mathcal{E}\subseteq \mathcal{V}\times\mathcal{V}$. Let $A$ denote the adjacency matrix with entry $a_{ji}\ne 0$ for $(v_i,v_j)\in\mathcal{E}$ in directed graphs. For undirected graphs, $a_{ij}=a_{ji}\ne0$ when $(v_i,v_j)\in\mathcal{E}$ or $(v_j,v_i)\in\mathcal{E}$. The degree matrix is followed as $D:=diag\{d_1,\ldots,d_n\}:=\{d_i|d_i=\sum_{j\in \mathcal{N}(i)}a_{ij}\textit{ or }\sum_{j\in \mathcal{N}(i)}|a_{ij}|\}$, where $\mathcal{N}(i)$ represents the set of adjacent agents of agent $i$.
For a signed network, to guarantee the consensus, the diagonal entries of Laplacian matrix are always positive. The corresponding dynamics is
\begin{equation}\label{ab}
  \!\dot x_i(t)\!\!=\!-\!\!\!\!\!\!\sum\limits_{j\in \mathcal{N}(i)}\!\!\!(|a_{ij}|x_i(t)-a_{ij}x_j(t))+b_iu_i(t),i\!=\!1,\ldots,n,
\end{equation}
where $u_i(t)$ is the external input for $v_i$ and $b_i=1$ if node $i$ is a leader, otherwise $b_i=0$. The term $|a_{ij}|$ ensures that every feedback to agent $i$ is negative, so that the system is stable. Different from consensus, it is not necessary to take positive magnitude on the diagonal entries of Laplacian matrix when discussing controllability. Thus, another Laplacian dynamics is arisen as follows.
\begin{equation}\label{nonab}
 \dot x_i(t)\!=\!-\!\!\!\!\sum\limits_{j\in \mathcal{N}(i)}a_{ij}(x_i(t)-x_j(t))+b_iu_i(t),i=1,\ldots,n.
\end{equation}
The compact form of \eqref{ab} is
\begin{equation}\label{protocol}
  \dot x=-Lx+Bu,
\end{equation}
where $L=D-A$, the entries of $L$ can be represented as
\begin{equation}
{l_{ij}}= \left\{ \begin{array}{*{20}{c}}
{-a_{ij},} & \textrm{$i\ne j;$}\\
 \sum\limits_{j\in \mathcal{N}(i)}|a_{ij}|,& \textrm{$i=j.$}
\end{array} \right.
\end{equation}
Similarly, dynamics \eqref{nonab} can be rewritten as
\begin{equation}
  \dot x=-L^*x+Bu,
\end{equation}
where $L^*=D-A$, the entries of $L^*$ are expressed by
\begin{equation}
{l_{ij}^*}= \left\{ \begin{array}{*{20}{c}}
{-a_{ij},} & \textrm{$i\ne j;$}\\
{\sum\limits_{j\in \mathcal{N}(i)}{a_{ij} },}& \textrm{$i=j.$}
\end{array} \right.
\end{equation}
\begin{remark}
The following will prove that the controllability under protocol \eqref{ab} differs from that under protocol \eqref{nonab}, while the structural controllability is unified under both protocols, as well as the strong structural controllability.
\end{remark}
\section{Controllability of unweighed graphs}
\subsection{Controllability Under Different Protocols}
\begin{definition}
For a circle with even number of nodes and half weights of positive, if the positive and negative weights are in alternant order or there are only two adjacent pairs of positive and negative weights, then we call this circle zero circle.
\end{definition}
\begin{figure}[!h]
  \centering
    \subfigure[Antagonistic weights in adjacent order;]{\label{adgraph}
  \includegraphics[width=1in]{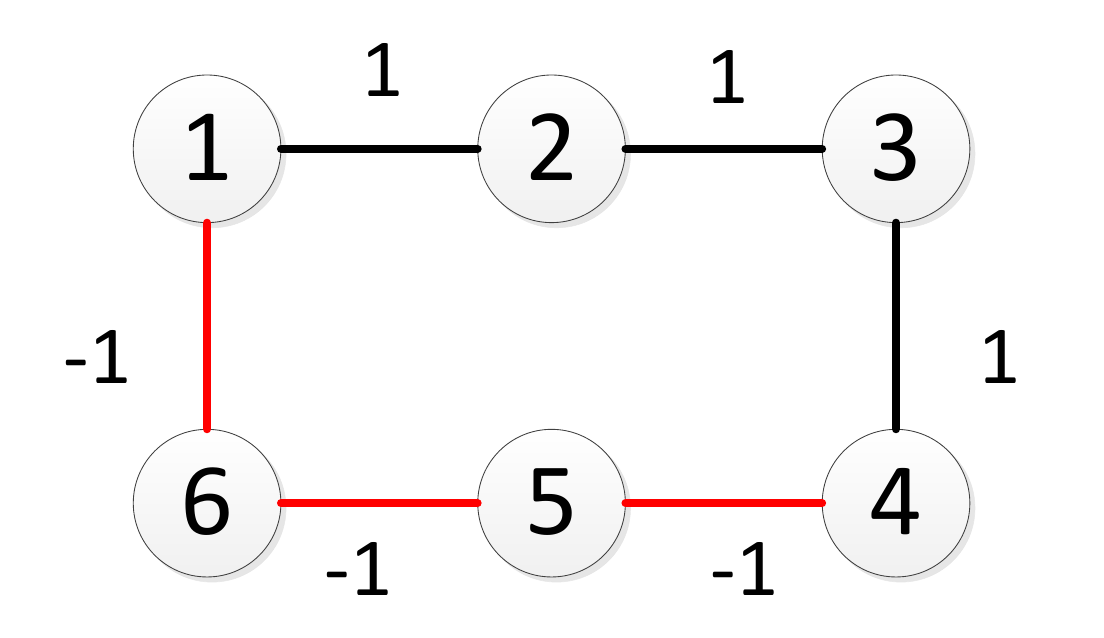}}\hspace{0.4in}
  \subfigure[Antagonistic weights in alternant order.]{
  \label{algraph}
  \includegraphics[width=1in]{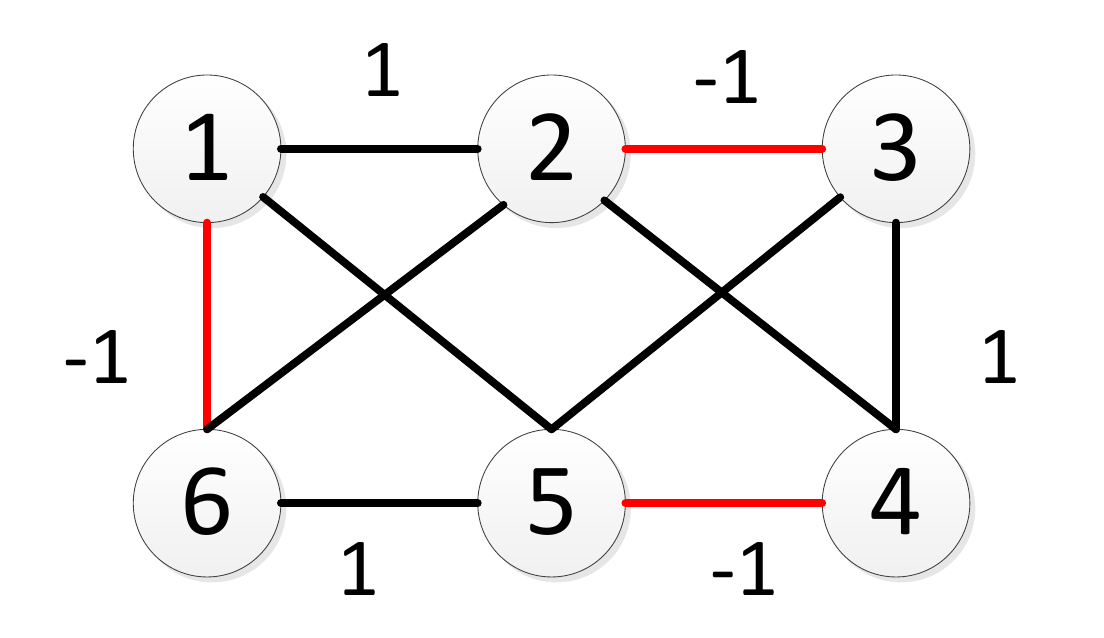}}\caption{Zero circles.}\label{twosub}
\end{figure}

For a zero circle in case that positive and negative weights are in alternant order, we allow nodes of even order to be connected with the same weight, so is to nodes of odd order. Two cases of zero circles are presented in Fig. \ref{twosub}, where the red line represents negative weight.

\begin{lemma}\label{zero}
The multiplicity of eigenvalue zero associated with $L^*$ is two for a zero circle graph.
\end{lemma}

\begin{proof}
In case that the positive and negative weights are in alternant order, the Laplacian matrix is
{\begin{footnotesize}\begin{equation*}
L^*=\left( {\begin{array}{*{20}{r}}
0&{ - 1}&0&0& \cdots &0&0&0&1\\
{ - 1}&0&1&0& \cdots &0&0&0&0\\
0&1&0&{ - 1}& \cdots &0&0&0&0\\
 \vdots & \vdots & \vdots & \vdots & \ddots & \vdots & \vdots & \vdots & \vdots \\
0&0&0&0& \cdots &{ - 1}&0&1&0\\
0&0&0&0& \cdots &0&1&0&{ - 1}\\
1&0&0&0& \cdots &0&0&{ - 1}&0
\end{array}} \right).
\end{equation*}\end{footnotesize}}
It can be derived that the even rows are linearly dependent, so is to the odd rows. Thus, $rank(L^*)=n-2$, and the multiplicity of eigenvalue zero is two. In case that there are only two adjacent pairs of positive and negative weights, the Laplacian matrix is
{\begin{scriptsize}\begin{equation*}
L^*\!=\!\left(\!\!\! {\begin{array}{*{20}{r}}
0&{ - 1}&0& \cdots &0&0&0&0&0& \cdots &0&1\\
{ - 1}&2&{ - 1}& \cdots &0&0&0&0&0& \cdots &0&0\\
 \vdots & \vdots & \vdots & \ddots & \vdots & \vdots & \vdots & \vdots & \vdots & \ddots & \vdots & \vdots \\
0&0&0& \cdots &{ - 1}&2&{ - 1}&0&0& \cdots &0&0\\
0&0&0& \cdots &0&{ - 1}&0&1&0& \cdots &0&0\\
0&0&0& \cdots &0&0&1&{ - 2}&1& \cdots &0&0\\
 \vdots & \vdots & \vdots & \ddots & \vdots & \vdots & \vdots & \vdots & \vdots & \ddots & \vdots & \vdots \\
1&0&0& \cdots &0&0&0&0&0& \cdots &1&{ - 2}
\end{array}} \right),
\end{equation*}\end{scriptsize}}
where the diagonal entries of the first row and the $(n/2+1)$th row are both 0. Let $\bar L$ denote the matrix obtained from $L^*$ by deleting the first and the $(n/2+1)$th rows. Hence, $rank(\bar L)=n-2$, and there are two linearly independent solutions for $\bar Lx=0$. In addition, by the structure of $\bar L$, there is $x_i=x_{n-i}$, $i=2,\ldots,{(n/2-1)}$, where $x_i$ is a component of $x$. As a consequence, $Lx=0$ still holds.
\end{proof}

There are the other special cases which arise repeated eigenvalue zero. For Fig. \ref{identicial},
\begin{equation*}
L^*=\left(
  \begin{array}{rrrrrr}
    -3 & 0 & 0 & 1 & 1 & 1 \\
    0 & -3 & 0 & 1 & 1 & 1 \\
    0 & 0 & -3 & 1 & 1 & 1 \\
    1 & 1 & 1 & -1 & -1 & -1 \\
    1 & 1 & 1 & -1 & -1 & -1 \\
    1 & 1 & 1 & -1 & -1 & -1 \\
  \end{array}
\right),\end{equation*} \begin{equation*}L_s=\left(
  \begin{array}{cccccc}
-1 & -1 & -1 \\
-1 & -1 & -1 \\
-1 & -1 & -1 \\
  \end{array}
\right).
\end{equation*}
It can be seen that $L_s$ is a submatrix of $L^*$, with $rank(L_s)=1.$ The multiplicity of eigenvalue zero associated with $L_s$ is two, and the corresponding eigenvectors $x_s$ are $\{1,-1,0\}^T$ and $\{1,0,-1\}^T$. Moreover, $\{0,0,0,0,1,-1,0\}^T$ and $\{0,0,0,0,1,0,-1\}^T$ are the solutions of $L^*x=0$. 
\begin{figure}[!h]
  \centering
  \includegraphics[width=0.9in]{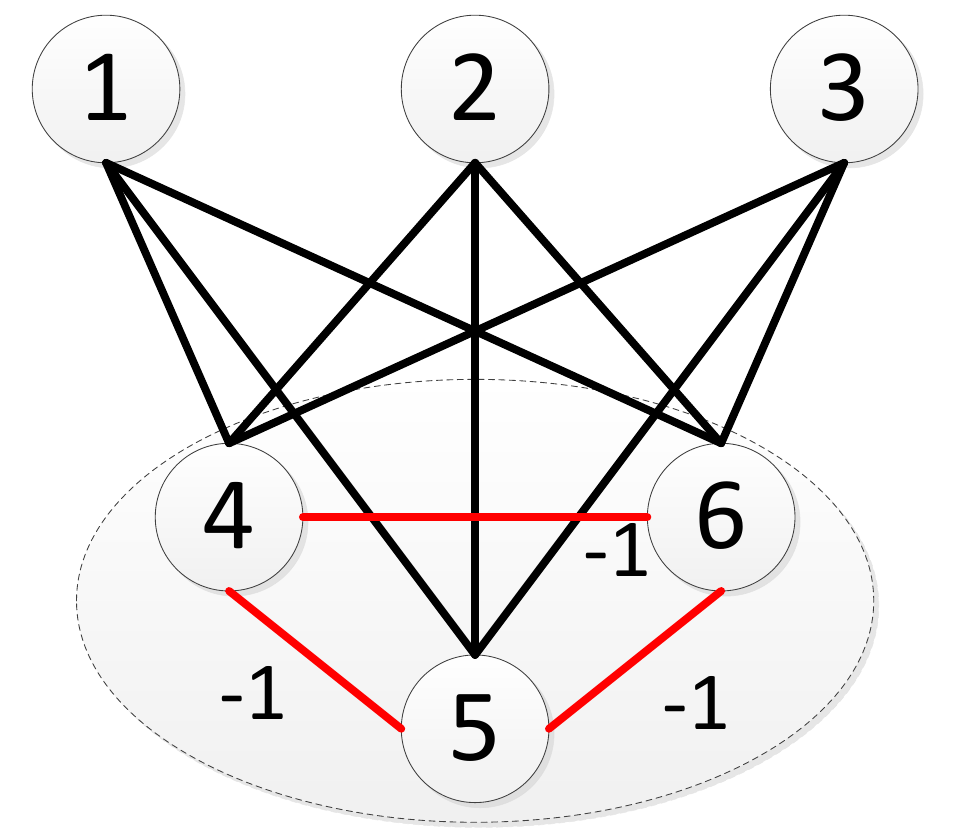}\\
  \caption{Identical nodes.}\label{identicial}
\end{figure}
\begin{definition}
 $\mathcal{G}$ is said to be a complete graph if each node of $\mathcal{G}$ is connected with all the other nodes.
\end{definition}
\begin{definition}
 Let $\mathcal{G}_s$ be a complete subgraph of $\mathcal{G}$, with all weights being the same value $\alpha$  and $\mathcal{V}_s\subset\mathcal{V}$. If any two nodes in $\mathcal{V}_s$ are connected with the same $|\mathcal{V}_s|$ nodes in $\mathcal{V}/\mathcal{V}_s$ and the weights of the corresponding edges are the identical value $-\alpha$ , then we say those nodes in $\mathcal{G}_s$ are identical nodes.
\end{definition}
\begin{definition}
Let $\mathcal{V}_t$ denote a  set of nodes whose each node is connected with nodes $i$ and $j$.
If $a_{ik}=-a_{jk}$ for $\forall k\in\mathcal{V}_t$, we call this class pair of nodes $i,j$ opposite pair.
\end{definition}

The opposite pairs also take more zero eigenvalues rising in $L^*.$ For the graph in Fig. \ref{opposite}, there are
\begin{equation*}
 L^*= \left(
     \begin{array}{*{5}{r}}
       -2 & 1 & 1 & 0 & 0 \\
       1 & -2 & 1 & 1 & -1 \\
       1 & 1 & -2 & -1 & 1 \\
       0 & 1 & -1 & 1 & -1 \\
       0 & -1 & 1 & -1 & 1 \\
     \end{array}
   \right), L_p= \left(
     \begin{array}{*{5}{r}}
1 & -1 \\
-1 & 1 \\
     \end{array}
   \right).
\end{equation*}
It can be seen that $L_p$ is a submatrix of $L^*$. If $x_p=\{1,1\}^T$, there is $L_px_p=0$. Meanwhile, if $x=\{\mathbf{0},x_p^T\}^T=\{0,0,0,1,1\}^T$, there is $L^*x=0.$  
\begin{figure}[h]
  \centering
  \includegraphics[width=0.9in]{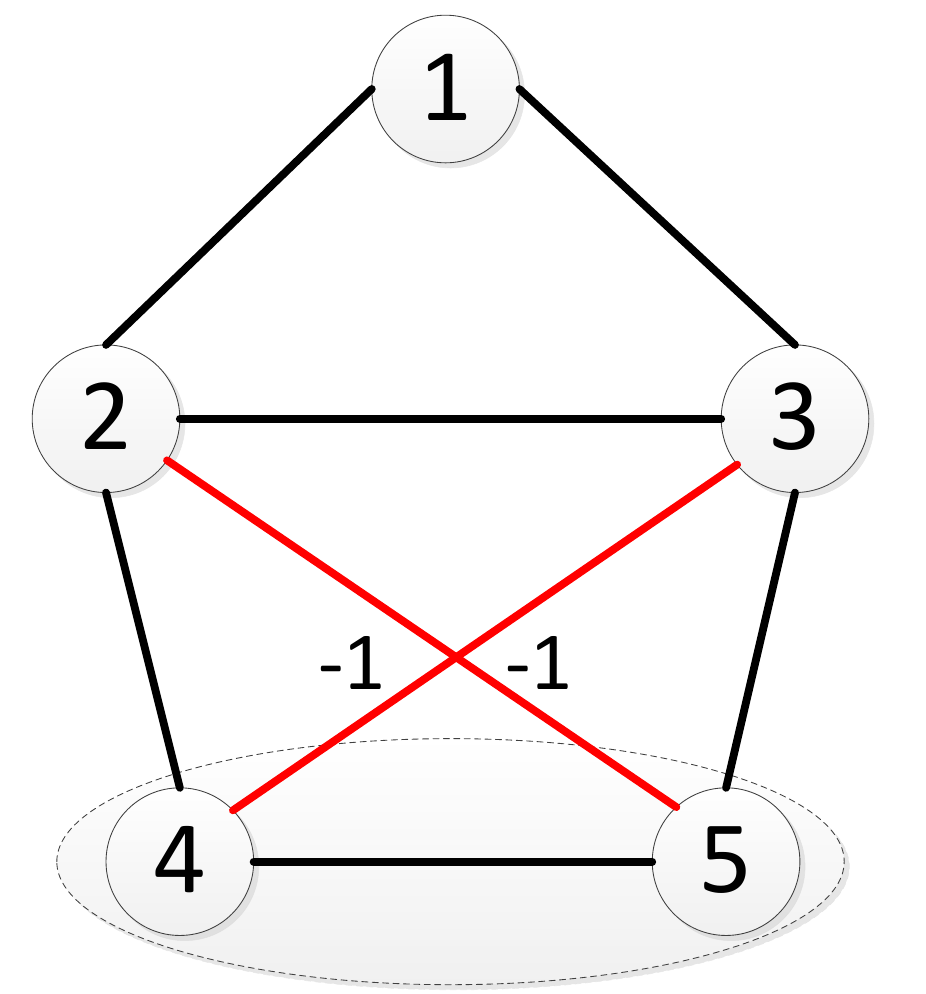}\\
  \caption{A pair of opposite nodes.}\label{opposite}
\end{figure}
\begin{lemma}\label{zeroc}
The multiplicity of eigenvalue zero associated with $L^*$ is $k$ if and only if $k$ is the sum of the numbers of zero circles, opposite pairs and identical nodes in $\mathcal{G}.$
\end{lemma}
\begin{proof}
Let us first prove the ``if'' part.  Assume that there are $k_1$ zero circles in a connected graph, then there are some nodes of zero circles connected with more than two nodes. Let $Z_i$ denote the node set of zero circle $z_i$, and $M=V\setminus \bigcup_{i=1}^{k_1} Z_i$. There is no circle consisting of both the nodes of $Z_i$ and $M$. Thus, for the first case that positive and negative weights are in alternant order, it can be derived $x_i=x_j=x_q$ for $L^*x=0$, where $(i,q)\in\mathcal{E}$, and node $j$ is reachable from node $i$ with $i,j\in M$, $q\in Z_i$. By Lemma \ref{zero}, if there are $k_1$ zero circles, there are $k_1+1$ linearly independent solutions of $L^*x=0$. For the second case of adjacent pairs of antagonistic weights, there is also $x_i=x_j=x_q$ for $L^*x=0$, where $(i,q)\in\mathcal{E}$, node $j$ is reachable from node $i$, and $i,j\in M$, $q\in Z_i$. If there are $k_2$ identical nodes, then there are another $k_2-1$ zero eigenvalues of $L^*.$ In addition, $k_3$ opposite pairs mean $k_3$ individual zero eigenvalues. Hence, there are $k=k_1+k_2+k_3$ zero eigenvalues.

For the part ``only if'', assume that the structures of zero circles are destructed. Then, for the first case, there is an additional connection between even order node and odd order node. If $Lx=0$, the components of $x$ corresponding to the zero circle are the same, that is, $rank(L)=n-1.$ For the second case, if $Lx=0$ with $x_m=x_h$, $(m,h)\in\mathcal{E}$ and $m,h\in Z_i$, then all the entries of $x$ corresponding to $Z_i$ are the same. Once the structure of zero circle is destructed, there exists one node $j\in Z_i$ with $|\mathcal{N}(j)\bigcap Z_i|>2$. Based on $x_p=x_{|Z_i|-p}$, there is  $x_j=x_{\mathcal{N}(j)\bigcap Z_i}$. Let the $m$th node and the $j$th node of $Z_i$ be connected, with $j>m$. Then, $x_m=x_{m+1}=\cdots=x_{j-1}=x_{j}$, and accordingly all the entries of $x$ corresponding to $Z_i$ are the same. The above cases do not consider identical nodes and opposite pairs. For identical nodes and opposite pairs, once the links or weights are changed, the solutions of $L_sx_s=0$ and $L_px_p=0$ do not coincide the principle of $Lx=0.$
\end{proof}
\begin{remark}
It is worthy to note that $|Z_i\bigcap Z_j|\in\{0,1\}$ which means that two zero circles can share at most one common node. Besides, $k_1,k_2,k_3>0.$
\end{remark}

It is well known that the system with repeated eigenvalues requires more inputs to ensure controllability than the system with simple ones. Since zero circles, identical nodes and opposite pairs take more zero eigenvalues, these structures are disadvantage for controllability. Note that for the Laplacian $L$ under dynamics \eqref{ab}, the eigenvalue zero is always simple \cite{biconsensus}.
{\begin{theorem}\label{subspace}
Under dynamics \eqref{nonab}, the system is uncontrollable if $q<k$, where $q$ is the number of inputs and $k$ is the sum of the numbers of zero circles, opposite pairs and identical nodes.
\end{theorem}
\begin{proof}
Assume that the sum of the numbers of zero circles, opposite pairs and identical nodes is $k$, and there are $q$ inputs. By PBH test, only if $rank[\lambda I+L^*|B]=n$ for $\forall \lambda\in\mathbb{C}$, the system is controllable. It follows from Lemma \ref{zeroc} that $rank[\lambda I+L^*]=n-k$ for $\lambda=0$. Then, by elementary transformation, there is \[U^{-1}[\lambda I+L^*]U=\left( {\begin{array}{*{20}{c}}
{{\Lambda_{n - {k}}}}&0\\
0&0
\end{array}} \right),\]
where $\Lambda_{n - {k}}$ is an upper triangle matrix. Hence, $rank[\lambda I+L^*|b]=rank[U^{-1}[\lambda I+L^*]U|U^{-1}b]$ for single input $b$. And $rank[U^{-1}[\lambda I+L^*]U|U^{-1}b]= n-k+1$ if there is one nonzero entry of $U^{-1}b$ corresponding to the zero row of $U^{-1}[\lambda I+L^*]U$. As a consequence, for multiple inputs $B\in\mathbb{R}^{n\times q}$ with $q<k$, if the nonzero entries  of $U^{-1}B$ correspond to distinct zero rows of $U^{-1}[\lambda I+L^*]U$, there is $w= n-\sum_{i=1}^{i=p}(k-q)<n$, where $w$ is the dimension of controllable subspace.
\end{proof}}{
\begin{definition}[Altafini\cite{biconsensus}]
A signed graph $\mathcal{G}$ is said to be structurally balanced if it admits a bipartition of the nodes $\mathcal{V}_1$, $\mathcal{V}_2$, with $\mathcal{V}_1\cup\mathcal{V}_2=\mathcal{V},$ $\mathcal{V}_1\cap\mathcal{V}_2=\emptyset$, such that $\forall v_i,$ $v_j\in \mathcal{V}_q$ $(q\in\{1,2\}),a_{ij}\geq 0$, and $\forall v_i\in \mathcal{V}_q,$ $v_j\in\mathcal{V}_r,$ $q\neq r$ $(q,r\in\{1,2\}),$ $a_{ij}\leq 0$. It is said to be structurally unbalanced otherwise.
\end{definition}}
\begin{lemma}[Altafini\cite{biconsensus}]\label{sc}
A connected signed graph $\mathcal{G}$ is structurally balanced if and only if any of the following equivalent conditions holds:
\begin{enumerate}[1)]
  \item all cycles of $\mathcal{G}$ are positive;
  \item $\exists C=diag\{\sigma_1,\ldots,\sigma_n\}$, $\sigma_i\in\{\pm 1\}$ such that $CAC$ has all nonnegative entries;
  \item zero is an eigenvalue of $L$.
\end{enumerate}
\end{lemma}
\begin{proof}
1) $\Leftrightarrow$ structural balance, 2) $\Rightarrow$ structural balance, 1) and 2) $\Rightarrow$ 3), 3) $\Rightarrow$ 2), 2) $\Rightarrow$ structural balance $\Rightarrow$ 1). Hence, 1) $\Leftrightarrow$ 2) $\Leftrightarrow$ 3) $\Leftrightarrow$ structural balance.

The part 1) $\Rightarrow$ 3) has been verified in \cite{biconsensus}.

[ 1) $\Leftrightarrow$ structural balance] If the negative weights in circles are pairwise arisen, then these circles are positive. Assume that there are $p$ nodes in one positive circle, such that $a_{i_1i_2}a_{i_2i_3}\cdots a_{i_{p-1}i_p}$ $a_{i_pi_1}>0$, and $a_{k(k+1)}<0,$ $a_{h(h+1)}<0$. The nodes of circle can be partitioned into two parts, $v_i\in\mathcal{V}_1$, $i=k+1,\ldots,h$, and $v_{j}\in\mathcal{V}_2$, $j=h+1,\ldots,k$. The case of one pair of negative weights is the same as that of more than one pair of negative weights. Suppose that circles of structurally unbalanced $\mathcal{G}$ are positive. Since the topology is structurally unbalanced, the number of negative weights in one circle is odd, which is a contradiction to positive circle.

[ 2) $\Rightarrow$ structural balance] If $A$ is nonnegative, then all of nodes belong to $\mathcal{V}_1$, and $\mathcal{V}_2=\emptyset$, with all circles being positive. Assume that there exists a diagonal matrix $C=diag\{\sigma_1,\ldots,\sigma_n\}$, $\sigma_i\in\{\pm 1\}$ such that $CAC$ is nonnegative. Take $l$ nodes into one set, $v_i\in\mathcal{V}_1$, $i\in N_l=\{n_1,\ldots,n_l\}$, and the others into another set, $v_j\in\mathcal{V}_2$, $j\in N_{\neg l}=\{n_{l+1},\ldots,n_n\}$. Let the diagonal entries of $C$ corresponding to $\mathcal{V}_1$ be $-1$, and the others be $1$. $r_i=\{a_{i1},a_{i2},\ldots,a_{in}\},$ $i=n_1,n_2,\ldots,n_l$, $c_j=\{a_{1j},a_{2j},\ldots,a_{nj}\},$ $j=n_1,n_2,\ldots,n_l$, where $a_{ij}$ is the entry of $A$. Since $CAC=\bar A\geq0$, then $\bar r_i=-r_i\geq0$, $\bar c_j=-c_j\geq0.$ If $i=j$; $i,j\in N_l$, then $a_{ij}=0$. For the case of $i\neq j$; $i,j\in N_l$, there is $(-1)\cdot a_{ij}\cdot(-1)\geq 0$. For $i\neq j$; $i\in N_l,$ $j\in N_{\neg l}$ or $i\neq j$; $i\in N_{\neg l},$ $j\in N_l$, there is $(-1)\cdot a_{ij}\geq 0$.  It follows from the above that if $i\neq j$; $i,j\in N_l$, $a_{ij}\geq 0$. And if $i\neq j$; $i\in N_l,$ $j\in N_{\neg l}$ or $i\neq j$; $i\in N_{\neg l},$ $j\in N_l$, $a_{ij}\leq 0$. Once $i\neq j$; $i,j\in N_{\neg l}$, $a_{ij}\geq 0$.

[ 2) $\Rightarrow$ 3)] Since $\exists C=diag\{\sigma_1,\ldots,\sigma_n\}$, $\sigma_i\in\{\pm 1\}$ such that $CAC$ is nonnegative, then $CLC=\bar L$, $L\sim \bar L$. Thus, $L$ and $\bar L$ share a common eigenvalue zero.

[ 3) $\Rightarrow$ 2)] The topology $\mathcal{G}$ is connected if and only if the second smallest eigenvalue $\lambda_{2}(L)\ne0.$ Since zero is a simple eigenvalue of $L$ for connected $\mathcal{G}$, the corresponding eigenvector is $C1_n$. Based on $CLC1_n=\bar L1_n=0$, it can be derived that $CAC$ is nonnegative.
\end{proof}
\begin{remark}
The proof of Lemma \ref{sc} in \cite{biconsensus} is incomplete, and the above proof overcomes this shortcoming.
\end{remark}
\begin{definition}
Given the fixed topologies, if system $(L;B)$ is controllable for any choice of weights, we say $(\mathcal{G},\mathcal{V}_L)$ is controllable, where $\mathcal{V}_L$ is the leader set.
\end{definition}
\begin{theorem}\label{scontrollable}
Under structurally balanced topology and dynamics \eqref{ab}, $(\mathcal{G},\mathcal{V}_L)$ is controllable if and only if there exists one set of weights to render system controllable.
\end{theorem}
\begin{proof}
 By Lemma \ref{sc}, if one topology is structurally balanced, there is $CAC=\bar A$, where the nonzero entries in $\bar A$ are nonnegative. Thus, under structurally balanced topology, $CLC=\bar L$ always holds regardless of the weights in $A$. Based on $CLC=\bar L$, it can be derived that if we change the signs of some weights in $L$, $\bar L$ still shares common eigenvalues with $L$, and only the signs of the entries in eigenvectors are different. For structurally balanced topologies, if $(L;B)$ is controllable, $(\bar L;B)$ is always controllable. Consequently, $(\mathcal{G},\mathcal{V}_L)$ is controllable. If $(\mathcal{G},\mathcal{V}_L)$ is controllable, there is a choice of weights such that $(L;B)$ is controllable, where $\mathcal{V}_L$ is associated with leaders.
\end{proof}
\begin{remark}
Theorem 3 in \cite{sun} is a special case of Theorem \ref{scontrollable} in this paper. Even leaders are chosen from both $\mathcal{V}_1$ and $\mathcal{V}_2$, the controllability of $(L;B)$ and $(\bar L;B)$ is also equivalent.
\end{remark}
\subsection{Controllable Subspaces Under Different Protocols}
\begin{figure}[!h]
  \centering
\subfigure[Two cells;]{
\includegraphics[width=1.1in]{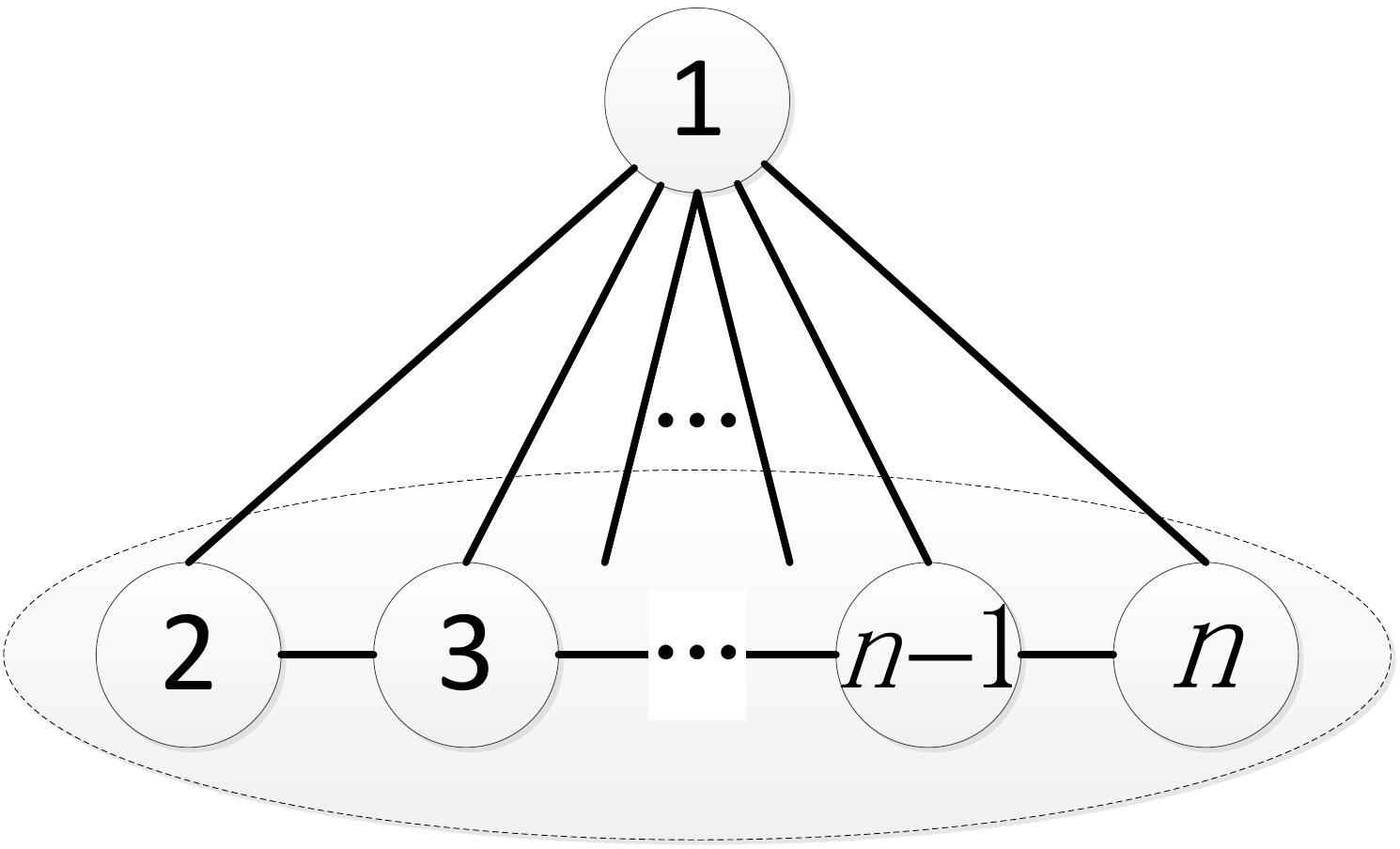}\label{tc}}\hspace{0.3in}
\subfigure[A star graph.]{
\includegraphics[width=1in]{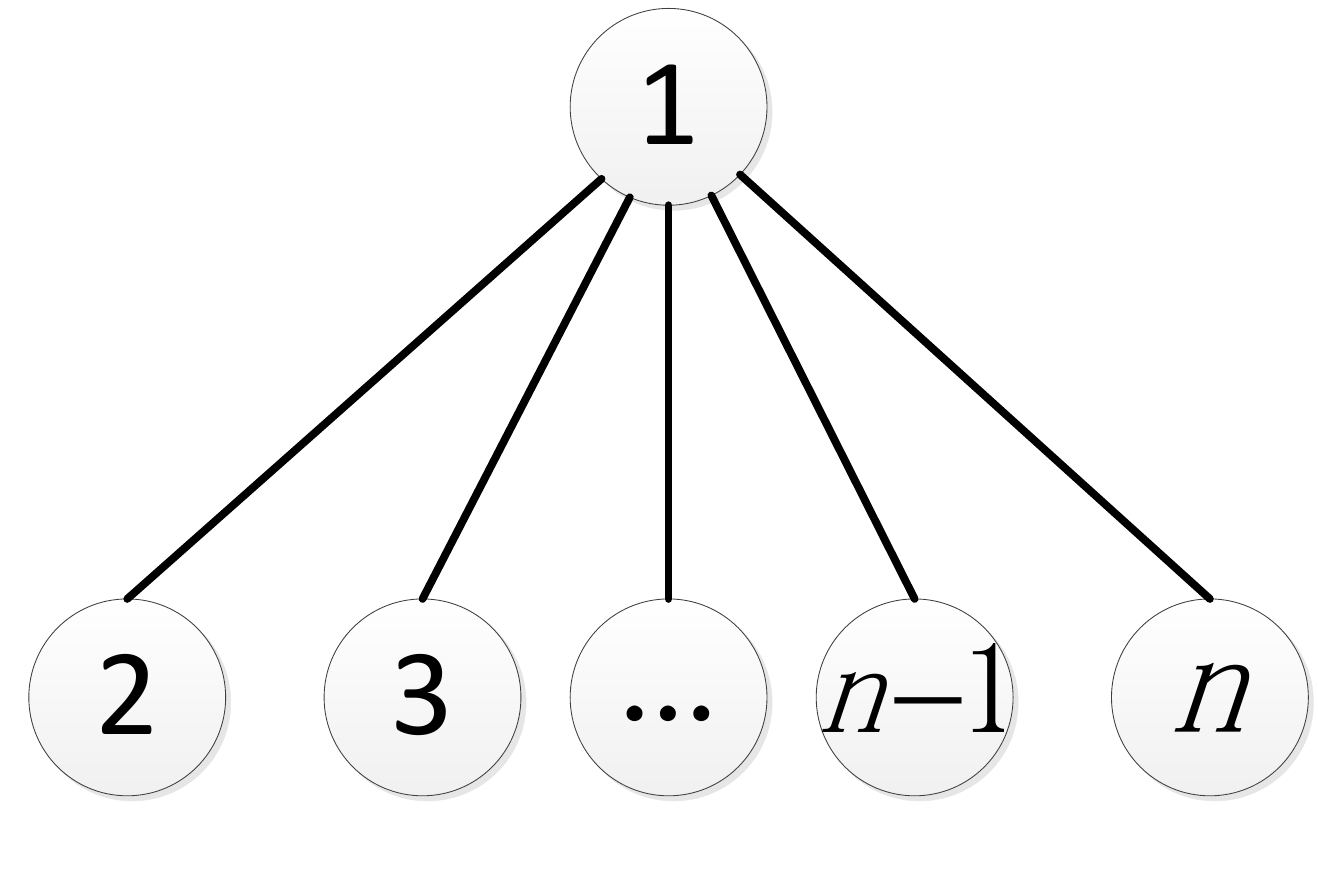}\label{star}}
\caption{Systems with the same controllable subspace.}
\end{figure}

For both Figs. \ref{tc} and \ref{star}, if the links from node 1 are weighted by the same value, the dimension of controllable subspace is determined under protocol \eqref{nonab}.  If the center node of a star graph is chosen as the single leader, then this star graph is called as a T-star graph.
{\begin{proposition}\label{dimensions2}
 For either of the following two scenarios
\begin{enumerate}[i)]
  \item under dynamics \eqref{nonab};
  \item under dynamics \eqref{ab} and structurally balanced topology,
\end{enumerate}
the dimension of controllable subspace is greater than two if and only if the connected graph is not expanded by a T-star graph.
\end{proposition}}
\begin{proof}
Assume that node 1 is chosen as the single leader, and all followers are connected with node 1, with the corresponding weights being identical. Then, the Laplacian matrix can be decomposed into the form as
\begin{equation*}
  L^*=\left(\begin{array}{*{20}{c}}
  \alpha(n-1)&\alpha 1_{n-1}^T\\
  \alpha 1_{n-1}&L_f
  \end{array}\right),
\end{equation*}
where $\alpha$ represents the weight between leader and followers. $1_{n-1}$ is an eigenvector of $L_f$, such that all the other eigenvectors of $L_f$ are orthogonal to $1_{n-1}$. Then, under dynamics \eqref{nonab}, there are $n-2$ eigenvectors $x_f$ with $1_{n-1}^Tx_f=0$. Hence, there is only one node associated with $L_f$ being controllable. Consequently, there are only two controllable nodes in the whole system. Namely, the dimension of controllable subspace is two.  Once the connected graph is not expanded by a star graph anymore, the dimension of controllable subspace is  greater than two. For an antagonistic network with structurally balanced topology, its controllability is equivalent to the one with only positive weights. Under dynamics \eqref{ab}, the weights $1$ and $-1$ take the same role on controllability. Hence, if the connected graph is expanded by a T-star graph, the dimension of controllable subspace is two.
\end{proof}

In view of Proposition \ref{dimensions2}, once the connected graph is expanded by a T-star graph, whatever the connection between followers is,
the dimension of controllable subspace always takes the same value under \eqref{nonab}. The structure between nodes $2$ to $n$ dose not affect the controllability.
\begin{remark}
 Proposition \ref{dimensions2} also applies to directed graphs.
\end{remark}
\section{{Unifying Structural Controllability Under Different Protocols}}
\subsection{Structural Controllability of Weighted Graphs}
\begin{lemma}[Lin\cite{lin}]\label{linl}
The following statements are equivalent.
\begin{enumerate}
 { \item The pair $(A,B)$ is structurally controllable.
  \item The graph of $(A,B)$ contains no unaccessible node and no dilation.
  \item The graph of $(A,B)$ is spanned by a cactus.}
\end{enumerate}
\end{lemma}

\begin{figure}[!h]
\centering
  \subfigure[stem]{
\includegraphics[width=0.85in]{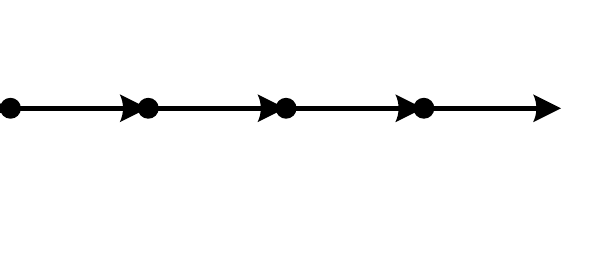}\label{stem}}
\subfigure[bud]{
\includegraphics[width=0.85in]{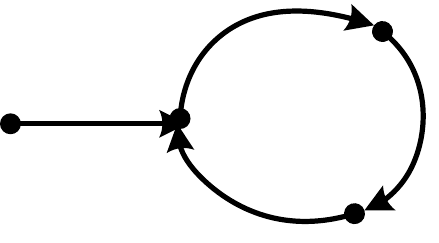}\label{bud}}
\subfigure[cactus]{
\includegraphics[width=0.9in]{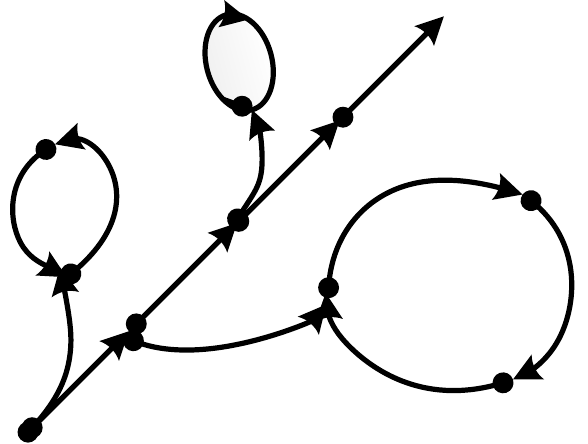}\label{cactus}
}\caption{The structures of stem, bud and cactus}\label{sbc}
\end{figure}

The structures of stem and bud are depicted as Fig. \ref{stem} and Fig. \ref{bud}, respectively. The union of stem and buds constitutes a structure of cactus  as shown in Fig. \ref{cactus}.
\begin{definition}[Lin\cite{lin}]
For a node set $S$, there is a neighbor set $T(S)$, such that there is an edge from each node in $T(S)$ to nodes in $S$. If $|S|>|T(S)|$, then we say that there is a dilation in the topology.
\end{definition}\begin{figure}[!h]
  \centering
  \subfigure[A simple multi-agent network;]{\includegraphics[width=0.9in]{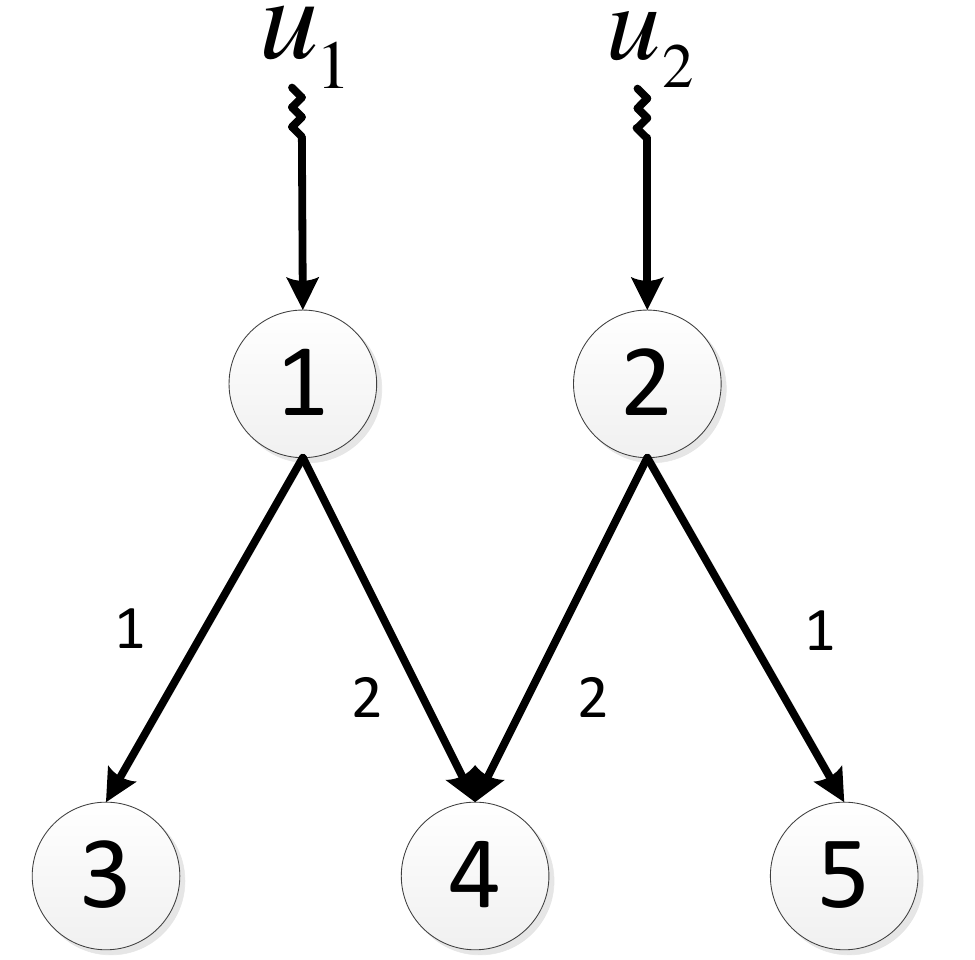}\label{nlinear}}\hspace{0.3in}
  \subfigure[A complex network with self-loop.]{\includegraphics[width=1.15in]{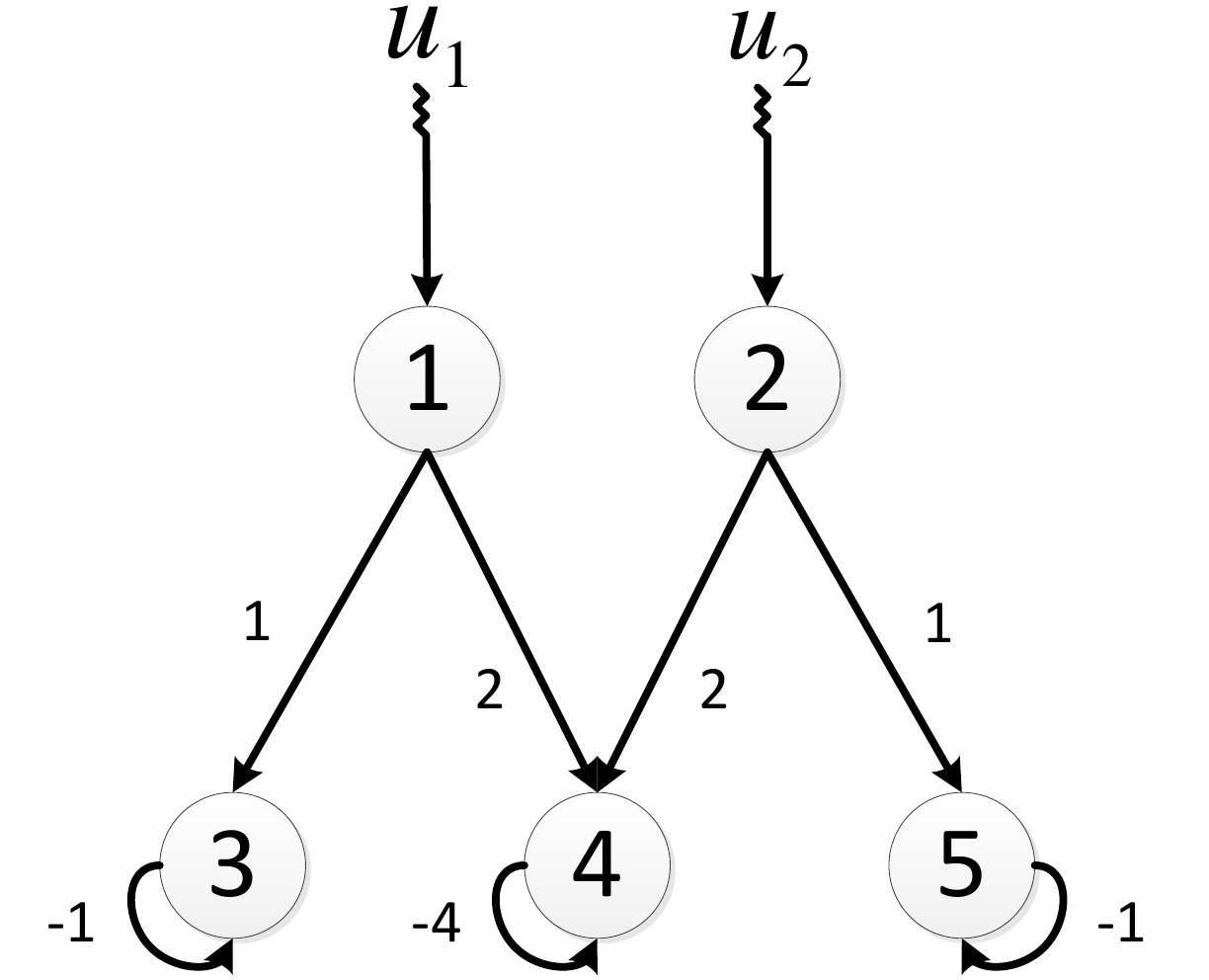}\label{complexnet}}\caption{Equivalent topologies to multi-agent system and complex network.}
\end{figure}

Actually, multi-agent systems have similar points with complex networks, the difference is the lack of degree matrix to complex networks. For a graph of multi-agent system as shown in Fig. \ref{nlinear}, there is an equivalent topology of complex network presented as Fig. \ref{complexnet}. {It can be found that the multi-agent system is a special case of complex networks with self-loop $l_{ii}=-A_i1_n$. For the special complex network of $l_{ii}=-A_i1_n$, there is not any dilation because of self-loops.} 
We say that one node is  accessible if there is a path from any input to the node, otherwise unaccessible.
\begin{figure}[!h]
  \centering
  \includegraphics[width=1.4in]{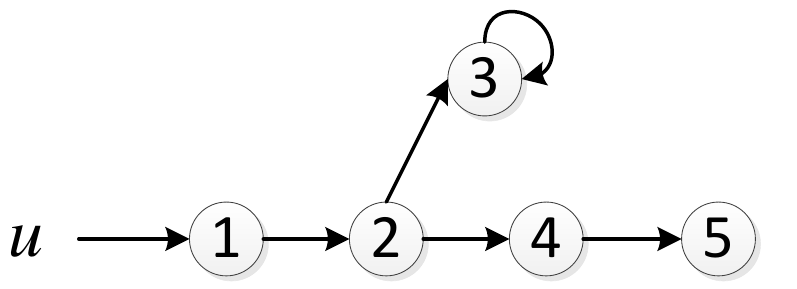}\\
  \caption{A cactus with branch}\label{dilation}
\end{figure}

From Fig. \ref{dilation}, if there is not a self-loop at node 3, there is a dilation. The complex network associated with dynamics \eqref{nonab} in Fig. \ref{dilation} is represented as
\begin{equation*}
  \dot x=\left(\begin{array}{*{23}{c}}
0&0&0&0&0\\
{{a_{21}}}&0&0&0&0\\
0&{{a_{32}}}&l_{33}&0&0\\
0&{{a_{42}}}&0&0&0\\
0&0&0&{{a_{54}}}&0
  \end{array} \right)x+\left(\begin{array}{*{23}{c}}
1\\
0\\
0\\
0\\
0
\end{array} \right)u.
\end{equation*}
The rank of $[\lambda I- A|B]$ is always five for both $\lambda=0$ and $\lambda\ne 0$.  In fact, there is only one free parameter in the third row of $A$. Let $l_{33}$ be the free parameter. The existence of $a_{32}$ only ensures the transmission of control information. If $l_{33}\ne 0$, then $a_{32}\ne0$. The value of $a_{32}$ does not affect the controllability.

For $(i,j)\in\mathcal{E}$, we call node $i$ as the father node of node $j$ and $j$ as the child node of node $i$. Once, node $i$ contains more than one child node, we say that there are branches at node $i$. If some nodes share one common father node, we call these nodes siblings, and the corresponding node set is denoted by $S_i$. Since there may be some nodes with more than one father node, we take $S_i\cap S_j=\emptyset$ with $i\ne j$. Namely, if there are some nodes contained by $S_i$, none of these nodes contained by another set $S_j$.
\begin{definition}
For a spanning tree with $m_i$ branches at each branch point, if each node of any $m_i-1$ branches contains a self-loop, we call this topology a pseudo spanning tree (PST).
\end{definition}

\begin{figure}[!h]
  \centering
  \includegraphics[width=1.4in]{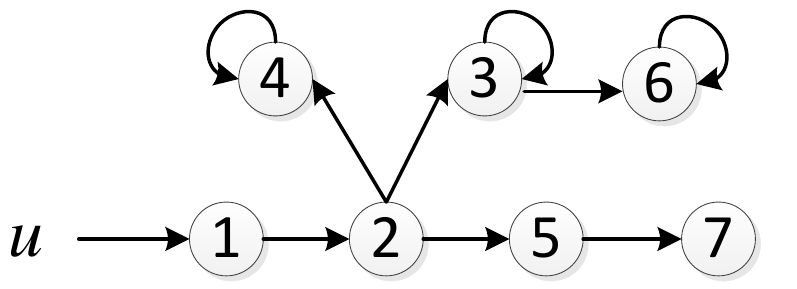}\\
  \caption{Branch lines with self-loop}\label{sib}
\end{figure}
For the system in Fig. \ref{sib}, if the weights of edges $e_{32}$ and $e_{42}$ take the same value, then $rank[\lambda I-A|B]=6$ for $\lambda=-a_{32}$, and $rank[\lambda I-A|B]=7$ for $\lambda\ne -a_{32}$. As for the branches at node 2, only if there are two branch lines with all nodes taking self-loop, and the weights of siblings take different values, $rank[\lambda I-A|B]$ can reach the maximum for $\lambda\ne0$.

\begin{lemma}\label{pst}
For a system with pseudo spanning tree, if the root node takes the leader role and the weights of siblings are different, then $rank[\lambda I-A|B]=n$ for $\lambda\ne0$.
\end{lemma}
\begin{proof}
The eigenvalues of $A$ associated with PST depend on the diagonal entries, that is, $\lambda_i=l_{ii}$ or $0$. Assume that there are $r$ branches, and each branch contains $|S_i|$ siblings. Then the weight for each one in the siblings takes different value. If not, there will arise one repeated eigenvalue $\lambda\ne 0$, such that $rank[\lambda I-A]<n-1$, and consequently $rank[\lambda I-A|B]\leq n-1$.
\end{proof}
\begin{remark}
The above arguments and Lemma \ref{pst} also apply to dynamics \eqref{ab}. The root node in directed tree graphs is a node that can reach any other node.
\end{remark}
\begin{lemma}[Aguilar\cite{cesar}]\label{cesarlemma}
For a connected graph $\mathcal{G}$, the dimension of controllable subspace $dim\langle L;B\rangle\geq r+1$, where $B=e_i$, $r=\max d(i,j)$, and $d(i,j)$ is the distance from node $i$ to node $j$; $i,j\in\mathcal{V}.$
\end{lemma}
{\begin{lemma}\label{iff}
Let $\mathcal{C}=[\begin{array}{*{6}{c}}B&-LB&\cdots&(-L)^{n-1}B\end{array}]$, $P=[\begin{array}{*{5}{c}}e_{i_1}^T&e_{i_2}^T&\cdots&e_{i_p}^T\end{array}]^T\in\mathbb{R}^{p\times n}$, where $e_{i_j}$ is a basic standard vector, $i_1,i_2,\ldots,i_p\in\{1,\cdots,n\}.$ Suppose that all eigenvalues of $L$ are simple, then $rank(P[\lambda I+L|B])=p$ for $\forall \lambda \in\mathbb{C}$ if and only if $rank(P\mathcal{C})=p$.
\end{lemma}\begin{proof}
Let $P=[\begin{array}{*{5}{c}}e_{i_1}^T&e_{i_2}^T&\cdots&e_{i_p}^T\end{array}]^T\in\mathbb{R}^{p\times n}$, where $e_{i_j}$ is a basic standard vector, $i_1,i_2,\ldots,i_p\in\{1,\cdots,n\}.$ Assume that $rank(P[\lambda I+L|B])<p$. Then, there exists a nonzero vector $q\in\mathbb{R}^{p}$ such that $q^TP[\lambda I+L|B]=0$, which means that $q^TP(-L)=\lambda q^TP$ and $q^TPB=0$. Hence, $q^TPB=q^TP(-L)B=q^TP(-L)^2B=\cdots=q^TP(-L)^{n-1}B=0$, $q^T[\begin{array}{*{6}{c}}B&-LB&\cdots&(-L)^{n-1}B\end{array}]=0.$ Once all eigenvalues of $L$ are simple, $\cup_{i=1}^n\ker([\lambda_i I+L|B]^TP^T)=\ker(C^TP^T)$. Consequently, $rank(P[\lambda I+L|B])=p,$ $\forall \lambda \in\mathbb{C}\Longleftrightarrow rank(P\mathcal{C})=p$, where $\mathcal{C}=[\begin{array}{*{6}{c}}\!\!B\!&\!-LB\!&\!\cdots\!&\!(-L)^{n-1}B\!\!\end{array}]$.
\end{proof}
\begin{theorem}\label{t2}
The systems \eqref{ab} and \eqref{nonab} under fixed topology are structurally controllable if and only if the topology $\mathcal{G}$ is connected and there are no unaccessible nodes.
\end{theorem}
\begin{proof}
The directed graphs are considered firstly. A directed connected graph $\mathcal{G}$ without unaccessible nodes can be spanned by a PST which is an union of cacti. For a general complex network with self-loops, every weight can be set arbitrarily, and accordingly the network is always structurally controllable. Specially, for the special complex networks associated with multi-agent systems, there is one constrain of $l_{ii}=-A_i1_n$.
 As a consequence, if there are $m$ nonzero entries in $A_i$, there are only $m-1$ free parameters. The similar  argument was also presented in \cite{para}. At a branch, we take $a_{ii}$ as a free parameter, while $a_{i(i-1)}$ is only used to transmit the information. By Lemma \ref{pst}, even for $m-1$ free parameters, $[\lambda I-A|B]$ is still full row rank for both $\lambda=0$ and $\lambda\ne0$. Hence, by Lemma \ref{linl}, the system is structurally controllable if and only if $\mathcal{G}$ is spanned by PST.

  Secondly, the case of undirected graphs is considered. Assume that $B=e_1$ and $r=d(1,m)$, where $d(1,m)$ covers nodes $1$ to $m$. Namely, nodes $1$ to $m$ are controllable from Lemma \ref{cesarlemma}. If there is a branch at node $1$ with sibling nodes $m$ and $m+1$, then the Laplacian matrix can be expressed as
\begin{equation*}
 L=\left( {\begin{array}{*{20}{c}}
*& \cdots &*&*&*& \cdots &*\\
 \vdots & \ddots & \vdots & \vdots & \vdots & \ddots & \vdots \\
{{l_{m1}}}& \cdots &{{l_{mm}}}&?&?& \cdots &?\\
{{l_{(m + 1)1}}}& \cdots &?&{{l_{(m + 1)(m + 1)}}}&?& \cdots &?\\
 \vdots & \ddots & \vdots & \vdots & \vdots & \ddots & \vdots \\
*& \cdots &*&*&*& \cdots &*
\end{array}} \right),
\end{equation*}
where the question mark can be a zero or nonzero entry. Once $l_{m1}\ne l_{(m+1)1}$, at least one choice of weights $?$ occurs $l_{mm}\ne l_{(m+1)(m+1)}$. Hence, the $m$th and $(m+1)$th rows of $[\lambda I+L|e_1]$ can be linearly independent, which means that node $m+1$ also can be controllable from Lemma \ref{iff}. Since node $m+1$ is controllable, by the same manner, there are another $d(m+1,j)$ controllable nodes, $j\notin\{1,\ldots,m\}$. If there is a branch between nodes $m+1$ to $j$, and the weights between father node and sibling nodes take different values from each other, then all the sibling nodes can be controllable. Iteratively, there is a choice of weights such that all nodes are controllable. If a system under $B=e_i$ is structurally controllable, the corresponding undirected graph is connected.
\end{proof}
\begin{remark}\label{point}
The need to rigorously prove that in some cases the results of complex networks are applicable to multi-agent systems has been neglected by some previous works. Lemma 7 also applies to $L^*$.
\end{remark}
\subsection{Strong Structural Controllability of Undirected Graphs}
The strong structural controllability can be regarded as a reinforcing case of structural controllability, which requires the system to be controllable for any choice of weights.
If the leaf node of a path takes leader's role, the system is controllable,  where the leaf node is a node whose degree is one. Besides, if two adjacent nodes of path graphs and circle graphs are chosen as leaders, the systems are also controllable. For the above three cases of controllable systems, whatever the weights are, the systems are always controllable. Let $\mathcal{Q}(\mathcal{G})=\{A\in\mathbb{R}^{n\times n}: A=A^T, \text{and for},i\ne j,a_{ij}\ne0\Leftrightarrow(i,j)\in\mathcal{E}\}$, where the entries of $A\in\mathcal{Q}(\mathcal{G})$ can take arbitrary real values.

\begin{definition}
A network with a structure described by the graph $\mathcal{G}$ is strongly structurally controllable, if for $\forall A\in\mathcal{Q}(\mathcal{G})$, the system $(L;B)$ is controllable.
\end{definition}

The path has advantages in the discussion of controllability. Once the leaf node is chosen as a leader, all nodes  are controllable whatever the weights are. Under the case that the leaf node takes leader's role and the system is uncontrollable, it arises a situation that there is at least one eigenvector $x=0$. 
Thus, the system is uncontrollable in terms of the control strategy of $n$ inputs.
By contradiction, if the leaf node is chosen as the leader, the system is controllable.
\begin{definition}
 For a path initiated from one leader node to one leaf node, we call it a control path.
\end{definition}
\begin{theorem}\label{pathcontrollable}
Under dynamics \eqref{ab} and \eqref{nonab}, the dimension of the controllable subspace of strong structural controllability is not less than the number of the nodes in the longest distance control path.
\end{theorem}
\begin{proof}
Assume that the longest distance control path is $\mathcal{P}$ with $m$ nodes. Partition the system matrix and input matrix,
\begin{equation*}
  L=\left(\begin{array}{*{20}{c}}
  L_P&L_{P\neg P}\\
  L_{P\neg P}^T&L_{\neg P}
  \end{array}\right),\hspace{1.5em}
  B=\left(\begin{array}{*{20}{c}}
  e_1&B_1\\
  0&B_2
  \end{array}\right),
\end{equation*}
where $L_P\in\mathbb{R}^{m\times m}$, $B_1=\mathbf{0}_{m\times(n-m)}$, $B_2=I_{n-m}$, $e_1$ is a standard basis vector. Since every node in $\mathcal{V}_{\neg P}$ is injected by individual inputs, the nodes except for ones in $\mathcal{P}$ are controllable. For simplification, we just identify the controllability of the subsystem
\begin{equation*}
  \dot x_{P}=-L_Px_{P}+\left( {\begin{array}{*{20}{c}}
{{e_1}}&{ - L_{P\neg P}}
\end{array}} \right)\left( {\begin{array}{*{20}{c}}
{{u_1}}\\
{{x_{\neg P}}}
\end{array}} \right).
\end{equation*}
The submatrix $L_P$ is associated with path $\mathcal{P}$ and the leader node in $\mathcal{P}$ can be regarded as a leaf node. Then, the subsystem
\begin{equation*}
    \dot x_{P}=-L_Px_{P}+e_1u_1
\end{equation*}
is controllable.  Since the subsystem associated with $x_P$ is controllable, the term $- L_{P\neg P}x_{\neg P}$ does not affect the controllability of the subsystem. Therefore, for the system
\begin{equation*}
  \dot x=-Lx+\left(\begin{array}{*{20}{c}}
  e_1\\
  0
  \end{array}\right)u_1,
\end{equation*}
the subsystem associated with $x_P$ is also controllable. For some special choices of weights, the dimension of controllable subspace will exceed the limit $m.$ Whatever the choice of weights is, the dimension of controllable subspace is not less than $m.$
\end{proof}
\begin{remark}
Even the diagonal entries of $L$ take arbitrary real values, Theorem \ref{pathcontrollable} still holds.
\end{remark}
\begin{corollary}\label{shortest}
The dimension of controllable subspace of strong structural controllability is not less than $\max\limits_{j\in\mathcal{V}_F}\min\limits_{i\in\mathcal{V}_L} d(i,j)+1$.
\end{corollary}

Corollary \ref{shortest} can be proved by the same manner as Theorem \ref{pathcontrollable}. $d(i,j)$ represents the distance between node $i$ and node $j$, and $\mathcal{V}_F$ denotes the follower set. By Theorem \ref{pathcontrollable} and Corollary \ref{shortest}, it is easy to construct a controllable system by exerting inputs on proper nodes. Once, the paths oriented from leaders cover all of nodes, the system is controllable.

\begin{remark}
The dimension of strongly structurally controllable subspace is the dimension of controllable subspace that can be reached by any choice of weights.
\end{remark}

Taking control paths can only produce sufficient conditions for controllability. As for path graphs, one father node has only one child node, so the connection matrix $H=\alpha\in\mathbb{R}$ is one-dimensional and full rank.  Different to the child nodes of leader nodes, the leader nodes and original inputs can be regarded as an unity of new inputs. Consequently, the child node is also controllable.  If the connection matrix $H\in\mathbb{R}^{m\times m}$ between father nodes and child nodes is full rank, then the new inputs consisting of father nodes and original inputs are nonlinear to child nodes. Thus, these child nodes are also controllable.
\begin{figure}[!h]
  \centering
  \includegraphics[width=1in]{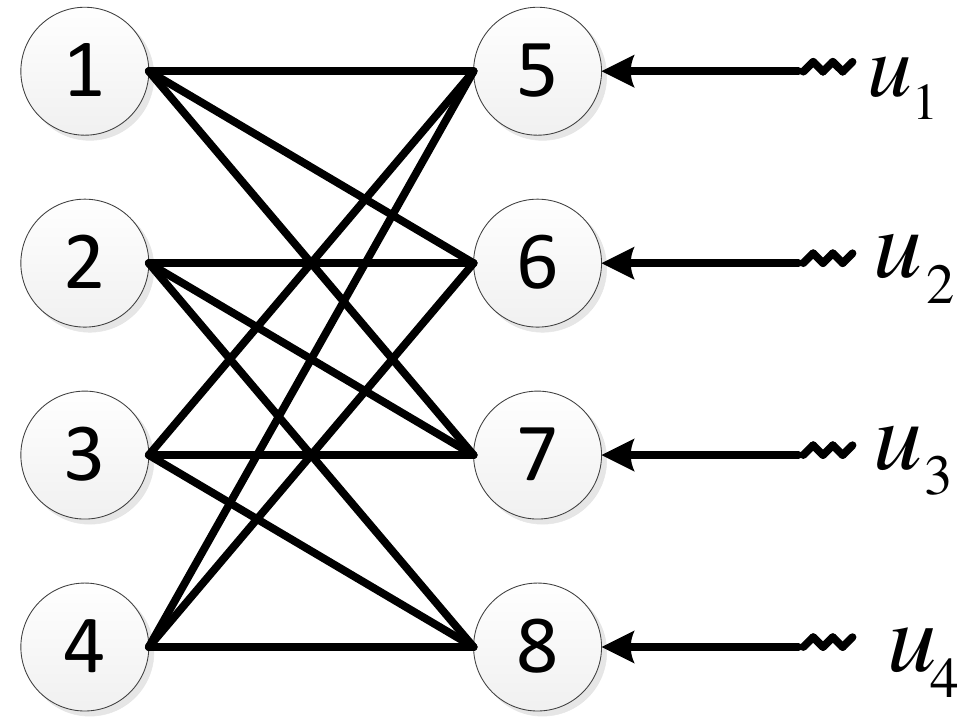}\\
  \caption{Strongly structurally controllable topology.}\label{zeroforce}
\end{figure}

 In particular, the graph in Fig. \ref{zeroforce} does not follow the feature of control paths, the corresponding system is still strongly structurally controllable. The connection matrix between nodes 1, 2, 3, 4 and nodes 5, 6, 7, 8 is
\begin{equation*}\label{connection}
  H=\left(\begin{array}{*{20}{c}}
  {*}&0&{*}&{*}\\
  {*}&*&{0}&{*}\\
  {*}&*&{*}&{0}\\
  {0}&*&{*}&{*}
  \end{array}\right),
\end{equation*}
with $rank(H)=4$, where $*$ represents an arbitrary real value.The rows of $H$ are linearly independent from each other, and the nonlinearity of $H$ brings nonlinear information flows to nodes 1, 2, 3, 4, such that the system is strongly structurally controllable even if $\mathcal{V}_1=\{5,6,7,8\}$ is not a zero forcing set. 

\begin{figure}[!h]
  \centering
  \includegraphics[width=1in]{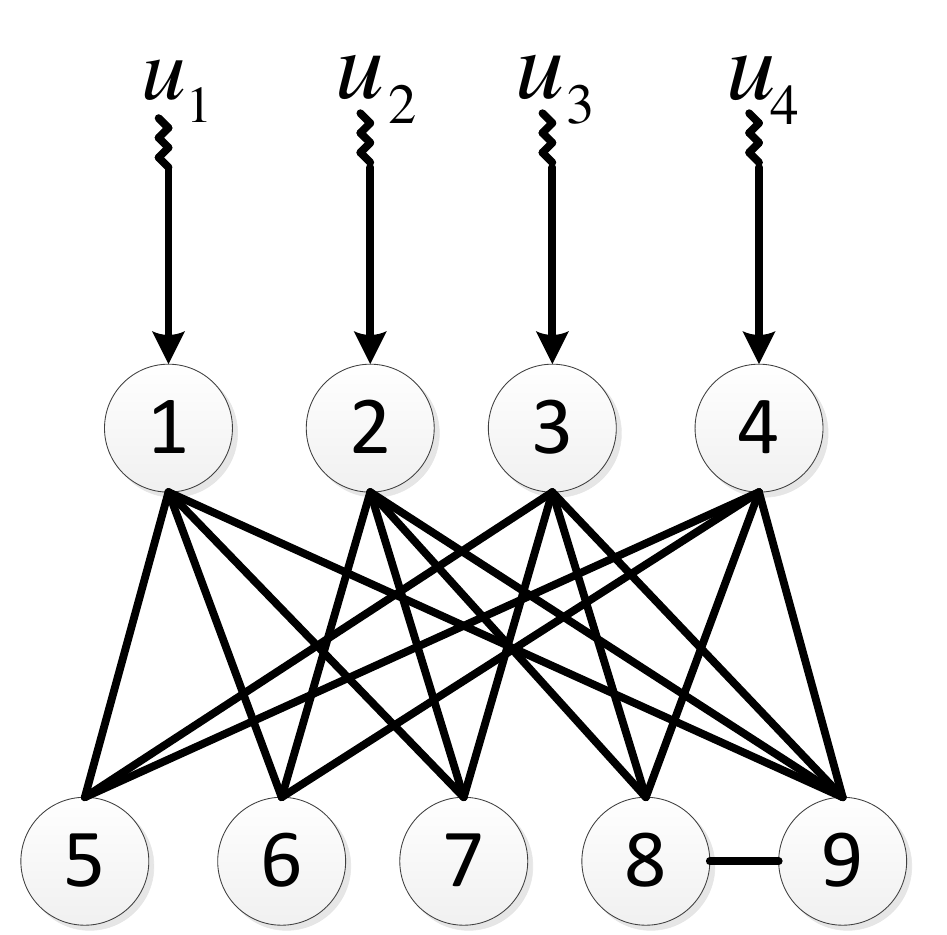}\\
  \caption{Four control nodes and five child nodes.}\label{nonrank}
\end{figure}
The Laplacian matrix of the graph in Fig. \ref{nonrank} is
\begin{equation*}\begin{aligned}
   L=& \left(\begin{array}{*{20}{c}}
  L_C&H^T\\
  H&L_{\neg C}
  \end{array}\right) \\
               =&\left( {\begin{array}{*{20}{cccc|ccccc}}
\textbf{*}&0&0&0&*&0&*&*&*\\
0&\textbf{*}&0&0&*&*&0&*&*\\
0&0&\textbf{*}&0&*&*&*&0&*\\
0&0&0&\textbf{*}&0&*&*&*&*\\
\hline
*&*&*&0&\textbf{*}&0&0&0&0\\
0&*&*&*&0&\textbf{*}&0&0&0\\
*&0&*&*&0&0&\textbf{*}&0&0\\
*&*&0&*&0&0&0&\textbf{*}&*\\
*&*&*&*&0&0&0&*&\textbf{*}
\end{array}} \right),\end{aligned}
\end{equation*}
where $H$ represents the connection matrix between leaders and followers.  $rank(H)=4$ means that the information flows injected into nodes 5, 6, 7, 8 and 9 are linearly dependent. 
Even the information flows oriented from leaders are not linearly independent, the subtopology of follower nodes can be constructed to produce enough linearly independent information flows, such that the system is strongly structurally controllable. 
It can be known that $rank(L_{\neg C}-\lambda I)\geq2$. Thus, $rank\left[\begin{array}{*{20}{c}}
  H&L_{\neg C}-\lambda I\\
  \end{array}\right]=5$, the system is strongly structurally controllable even if $\mathcal{V}_2=\{1,2,3,4\}$ is not a balancing set.

\begin{remark}
From the above arguments, the form of diagonal entries in $L$ does not affect the strong structural controllability. The strong structural controllability of dynamics \eqref{ab} is equal to that of dynamics \eqref{nonab}.
\end{remark}


If we change the labels of some nodes, the structure of topology is unchanged. We call these nodes symmetric nodes. The similar or same communication links induce symmetric nodes, which have the same dynamics because of symmetry.  Once a system contains symmetric nodes, we need to exert inputs on them to ensure strong structural controllability.
\begin{theorem}\label{twocaseofsc}
Systems \eqref{ab} and \eqref{nonab} are strongly structurally controllable if any of the following statements is true.
\begin{enumerate}[i)]
  \item There exists a choice of father nodes and child nodes such that the rank of the connection matrix between each pair of nodes equals to the number of child nodes and there is no symmetric follower node with the leaders being fixed.
  \item The graph under the fixed leaders contains no symmetric follower node and can be reduced into paths rooted from leaders.
\end{enumerate}
\end{theorem}
\begin{proof}
\begin{enumerate}[i)]
  \item Assume that all nodes are partitioned into $m$ parts, and the number of nodes in each part is $n_i$, $i=1,\ldots,m.$ Accordingly, Laplacian matrices $L$ and $L^*$ can be decomposed as
\begin{equation*}
L = \left( {\begin{array}{*{20}{c}}
{{L_1}}&{L_{21}^T}& \cdots &{L_{m1}^T}\\
{{L_{21}}}&{{L_2}}& \cdots &{L_{m2}^T}\\
 \vdots & \vdots & \ddots & \vdots \\
{{L_{m1}}}&{{L_{m2}}}& \cdots &{{L_m}}
\end{array}} \right),\end{equation*}\begin{equation*} L^* = \left( {\begin{array}{*{20}{c}}
{{\bar L_1}}&{\bar L_{21}^T}& \cdots &{\bar L_{m1}^T}\\
{{\bar L_{21}}}&{{\bar L_2}}& \cdots &{\bar L_{m2}^T}\\
 \vdots & \vdots & \ddots & \vdots \\
{{\bar L_{m1}}}&{{\bar L_{m2}}}& \cdots &{{\bar L_m}}
\end{array}} \right).
\end{equation*}
Without loss of generality, we set input matrix $B=\left[ {\begin{array}{*{20}{c}} B_1^T&B_2^T&\cdots&B_m^T\end{array}} \right]^T\in\mathbb{R}^{n\times n_1},$ with $B_1=I_{n_1}$, $B_i=0$, $i=2,\ldots,m.$ If $rank(L_{21})=rank(\bar L_{21})=n_2$, $rank(L_{32})=rank(\bar L_{32})=n_3$, $\cdots,$ $rank(L_{m\cdot{m-1}})=rank(\bar L_{m\cdot{m-1}})=n_m$, and there is not any symmetric follower node, then the rows of $L$ and $L^*$ associated with nodes $n_2$, $\ldots$, $n_m$ are linearly independent whatever the forms of $L_i$, $\bar L_i$ are. Consequently, $[L-\lambda I|B]$ and $[L^*-\lambda I|B]$ are both full row rank for any choice of weights, otherwise there exist symmetric follower nodes in the corresponding topology.
  \item If one graph under fixed inputs has no symmetric follower node and can be reduced into paths rooted from leaders, then the connection matrix between each pair of father and child nodes is full rank. Thus, the system is strongly structurally controllable.
\end{enumerate}
\end{proof}
{\begin{remark}
Zero forcing set in \cite{zeroforce1} is a special case of balancing set \cite{zeroforce2}, and moreover, balancing set is a special case of Theorem \ref{twocaseofsc}.
\end{remark}}

\section{Conclusions}
In this paper, we analyzed the controllability and structural controllability under two protocols. It was shown that special  structures-zero circles, identical nodes and opposite pairs arise more zero eigenvalues of $L^*$ than $L$, which requires more inputs to ensure controllability. For a structurally balanced topology, the controllable subspace remains unchanged even if the edge weights are altered under dynamics $L$. A sufficient and necessary condition for structural controllability of multi-agent system under both protocols was presented. Besides, we derived the sufficient conditions for strong structural controllability of multi-agent systems under both protocols, which indicates how the structures between child nodes and father nodes affect the strong structural controllability. In the future, we will consider the essential controllability which requires the system be controllable under any selection of leaders.


\end{document}